\documentclass[a4paper,reqno]{amsart}

\usepackage{amsmath, amssymb, amsthm, epsfig, esint}
\usepackage{hyperref, latexsym}
\usepackage{url}
\usepackage[mathscr]{euscript}
\usepackage[latin1]{inputenc}
\usepackage{color}
\usepackage{harpoon}
\usepackage{url}

\usepackage{setspace}
\usepackage{refcheck} 
\norefnames
\nocitenames

\def\today{\ifcase\month\or
  January\or February\or March\or April\or May\or June\or
  July\or August\or September\or October\or November\or December\fi
  \space\number\day, \number\year}

\newtheorem{theorem}{Theorem}
\newtheorem{conjecture}{Conjecture}
\newtheorem{lemma}[theorem]{Lemma}

\newtheorem{corollary}[theorem]{Corollary}
\theoremstyle{definition}

\theoremstyle{remark}

\newcommand{\mc}{\mathcal}

\newcommand{\F}{\mc{F}}

\newcommand{\G}{\mc{G}}

\newcommand{\K}{\mc{K}}

\newcommand{\C}{\mathbb{C}}

\newcommand{\R}{\mathbb{R}}

\newcommand{\Q}{\mc{Q}}
\newcommand{\Z}{\mathbb{Z}}

\newcommand{\p}{\varphi}

\newcommand{\dz}{\text{\rm d}z}
\newcommand{\dt}{\text{\rm d}t}

\newcommand{\dw}{\text{\rm d}w}
\newcommand{\dx}{\text{\rm d}x}

\newcommand{\dy}{\text{\rm d}y}

\renewcommand{\d}{\text{\rm d}}

\newcommand{\calc}[1]{\begin{align*}#1 \end{align*}}
\newcommand{\lcalc}[1]{\begin{align}\begin{split}#1\end{split} \end{align}}

\newcommand{\ov}{\overline}

\newcommand{\re}{{\rm Re}\,}

\newcommand{\ga}{\gamma}
\newcommand{\Dist}{{\rm Dist}}

\newcommand{\1}{\mathbf{1}}
\newcommand{\sg}{{\rm SG}}
\newcommand{\la}{\langle}
\newcommand{\ra}{\rangle}

\begin{document}


\title[]{A Sharpened Strichartz Inequality for Radial Functions}
\author[Gon\c{c}alves]{Felipe Gon\c{c}alves}
\date{\today}
\address{\noindent University of Alberta, Mathematical and Statistical Sciences, CAB 632, Edmonton, Canada T6G 2G1}
\email{felipe.goncalves@ualberta.ca}
\urladdr{sites.ualberta.ca/~goncalve}
\allowdisplaybreaks


\begin{abstract}
We prove a new sharpened version of the Strichartz inequality for radial solutions of the Schr\"odinger equation in two dimensions.  We establish an improved upper bound for functions that nearly extremize the inequality, with a negative second term that measures the distance from the initial data to Gaussians.
\end{abstract}


\maketitle


\section{Introduction}
Let $2\leq p,q \leq \infty$ satisfy $\frac{d}{p}+\frac{2}{q} = \frac{d}{2}$ and $(p,q,d)\neq (\infty,2,2)$. The Strichartz estimate for the Schr\"odinger equation (see \cite[Theorem 2.3]{Tao}) states that there exists a constant $C>0$ such that
\lcalc{\label{St-est}
\|\|u(x,t)\|_{L^p(\R^d,\dx)}\|_{L^q(\R,\dt)} \leq C \|f(x)\|_{L^2(\R^d,\dx)}
}
for all $f\in L^2(\R^d)$, where $u(x,t)$ is the solution of the Schr\"odinger equation in $\R^d$
\begin{equation*}
{\rm (SE)}\left\{
\begin{array}{lc}
\displaystyle \partial_t u(x,t) = i\Delta u(x,t), \\ 
u(x,0)=f(x).
\end{array}
\right.
\end{equation*}
Letting
\lcalc{\label{C-const}
C(p,q,d)=\sup_{f\neq 0} \frac{\|\|u(x,t)\|_{L^p(\R^d,\dx)}\|_{L^q(\R,\dt)} }{\|f(x)\|_{L^2(\R^d,\dx)}},
}
we say that a function $f\neq 0$ maximizes \eqref{St-est} if it realizes the supremum at \eqref{C-const}. It is conjectured that a function $f$ maximizes \eqref{St-est} if and only if it has the form $Ae^{-B|x|^2 + v\cdot x}$, where $A,B\in\C$, $\re B>0$ and $v\in\C^d$. If that is the case, it is then easy to show that $C(p,q,d)=[p^{-1/2p}2^{1/p-1/4}]^d$. 
\smallskip

This long-standing conjecture still is largely open, but for some few even exponents (where some extra structure emerges) it is known to be true. The first to prove this conjecture for $(p,q,d)\in\{(6,6,1);(4,4,2)\}$ was Foschi \cite{Fo}. Hundertmark and Zharnitsky \cite{HZ} also gave an alternative proof for these two cases. Later on, Carneiro \cite{Ca} and Bennett, Bez, Carbery and Hundertmark \cite{BBCH} gave other alternative proofs for these cases, including in addition the new case $(p,q,d)=(4,8,1)$. 

The present author also gave recently a new proof for all these exponents in \cite{Gon}, where the novelty was the use of orthogonal polynomial expansions to transform the desired sharp estimate into a series of finite-dimensional sharp inequalities for matrices that can be solved by generating functions techniques. It was already noticed in \cite[Appendix]{Gon} that the case $(p,q,d)=(4,4,2)$ is special in some way (for instance, the matrices that appear here are doubly stochastic) and that something more could be said in this situation.

It is worth mentioning that orthogonal polynomials have been used to produce sharp estimates in Harmonic Analysis in several instances.
The first most notorious and original use was in Beckner's thesis \cite{Be}, where he proved the sharp Hausdorff-Young inequality using Hermite polynomial expansions. More recently, Foschi \cite{F} used spherical harmonics and Gegenbauer polynomials in his proof of the sharp Tomas-Stein adjoint Fourier restriction inequality for the sphere. Later on, this strategy was extended by Carneiro and Oliveira e Silva \cite{CO} for other dimensions and even exponents.  Smoothing estimates for a general class of Schr\"odinger operators were also produced in \cite{BS} using Gegenbauer polynomials. 

Inspired by the work of Christ \cite{Jesus}, in the present paper we prove a sharpened version of the Strichartz inequality  for radial functions and exponents $(p,q,d)=(4,4,2)$ by performing a near-extremizer analysis that allow us to relate the distance from an extremizer (a Gaussian) to the inequality itself.

\subsection{Main results}
We will be only focused on the sharp Strichartz estimate \eqref{St-est} with exponents $(p,q,d)=(4,4,2)$ and for this reason we state it explicitly: If $u(x,t)$ solves {\rm (SE)} with initial data $f \in L^2(\R^2)$ then
\lcalc{\label{St-R2}
\|u\|_{L^4(\R^2\times\R)}\leq \frac{1}{\sqrt{2}} \|f\|_{L^2(\R^2)},
}
and equality is attained if and only if $f(x)=Ae^{-B|x|^2 + v\cdot x}$, where $A,B\in\C$, $\re B>0$ and $v\in\C^2$.

We say that a function $x\in\R^d \mapsto f(x)$ is radial if it depends only on $|x|$, where $|x|$ is the Euclidean norm of $x$. We denote by $L^2_{rad}(\R^d)$ the space of radial functions in $L^2(\R^d)$. Also, for a function $g\in L^2(\R^d)$ and a family of functions $\F\subset L^2(\R^d)$ we define
$$
\Dist_{L^2(\R^d)}(g,\F) = \inf \{\|g-f\|_{L^2(\R^d)}:f\in \F\}.
$$
We simply write $\Dist(g,\F)$ when it is clear by the context that this supposed to be calculated in $L^2(\R^d)$. We now state the main result of this paper, which is a sharpening of \eqref{St-R2}.

\begin{theorem}\label{radial-thm-R4}
Let $f\in L^2_{rad}(\R^2)$ and let $u(x,t)$ solve {\rm (SE)} with initial data $f$. There exists an universal constant $\gamma>0$ such that
\calc{
\|u\|_{L^4(\R^2\times\R)}\leq \frac{1}{\sqrt{2}} \|f\|_{L^2(\R^2)}\bigg[1-\gamma \frac{\Dist^2_{L^2(\R^4)}\left(f\otimes f,L^2_{rad}(\R^4)\right)}{\|f\|^4_{L^2(\R^2)}}\bigg]^{1/4},
}
where $f\otimes f(x,y)=f(x)f(y)$ for $(x,y)\in\R^2\times \R^2$. 
\end{theorem}

\noindent {\bf Remarks.}
\begin{enumerate}
\item We prove the above inequality with $\gamma=4\pi^{-2}$, however we believe it can be improved a little and it should hold with $\gamma=3/4$, which is best possible  (see the remark after Theorem \ref{radial-thm-R4-gen}).
\item  The result \cite[Theorem 1.3]{HZ} implies that
$$
\frac{1}{4} \|f\|^4_{L^2(\R^2)} - \|u\|^4_{L^4(\R^2\times\R)}  = \frac{1}{4} \Dist^2(f\otimes f, L^2_{u,v}(\R^4)) ,
$$
where $L^2_{u,v}(\R^4)$ is the subspace of functions invariant under rotations that fix the directions $u=(1,0,1,0)$ and $v=(0,1,0,1)$. In this way, our result can be interpret as 
$$
\Dist(f\otimes f, L^2_{u,v}(\R^4)) \gtrsim \Dist\left(f\otimes f,L^2_{rad}(\R^{4})\right)
$$
for radial $f$.
\end{enumerate}

In \cite[Theorem 1]{MC}, Christ shows a quantitative relation between the distance of $f\otimes f$ to the subspace of radial functions and the distance to radial Gaussians. It can be deduce from this result that
 \begin{equation*}
 \Dist_{L^2(\R^{2d})}\left(f\otimes f,L^2_{rad}(\R^{2d})\right) \approx \|f\|_{L^2(\R^d)} \Dist_{L^2(\R^d)}(f,\G), 
 \end{equation*}
where the implied constants (from above and below) depend only on the dimension $d$. Above 
$$\G=\{ae^{-b|x|^2}: a,b\in \C, \, \re b>0\}$$ is the space of radial Gaussians in $\R^d$. In particular, we obtain the following corollary.

\begin{corollary}
Let $f\in L^2_{rad}(\R^2)$ and let $u(x,t)$ solve {\rm (SE)} with initial data $f$. There exists an universal constant $\Gamma>0$ such that
\calc{
\|u\|_{L^4(\R^2\times\R)}\leq \frac{1}{\sqrt{2}} \|f\|_{L^2(\R^2)}\bigg[1-\Gamma \frac{\Dist^2_{L^2(\R^2)}\left(f,\G\right)}{\|f\|^2_{L^2(\R^2)}}\bigg]^{1/4}.
}
\end{corollary}
\noindent {\bf Remark.} The corollary begs the question whether this inequality holds as well for non-radial initial data $f$. We believe this to be true, but we have no formal proof.

\begin{theorem}\label{radial-thm-R4-gen}
Let $g:\R^2\times \R^2\to \C$ be a function in  $L^2(\R^2\times \R^2)$ such that $g(x,y)$ is radial in $x\in \R^2$ and $y\in \R^2$, that is, $g(x,y)$ depends only on $|x|$ and $|y|$. Let $u(x,y,t)$ solve {\rm (SE)} in $\R^2\times \R^2$ with initial data $u(x,y,0)=g(x,y)$. There exists an universal constant $\gamma>0$ such that
\lcalc{\label{michael-christ-ineq-gen}
\int_{\R} \int_{\R^2} |u(x,x,t)|^2\dx \dt \leq \frac{1}{4} \int_{\R^2\times \R^2} |g(x,y)|^2\dx\dy - \frac{\ga}{4} \Dist^2_{L^2(\R^4)}(g,L_{rad}^2(\R^4))
}
\end{theorem}

\noindent {\bf Remarks.}
\begin{enumerate}
\item Theorem \ref{radial-thm-R4}, with the same $\gamma$, follows directly from Theorem \ref{radial-thm-R4-gen} by taking $g(x,y)=f(x)f(y)$. We prove the above inequality with $\gamma=4\pi^{-2}$, but we believe it should hold with $\gamma=3/4$. The optimality  of $\gamma=3/4$ is supported by numerical computations presented in Section \ref{numerics}. This would be best possible since the function
$$
g(x,y) = (1-2\pi|x|^2)(1-2\pi|y|^2)e^{-\pi(|x|^2+|y|^2)} = \Psi_1(x)\Psi_1(y),
$$ 
attains equality in \eqref{michael-christ-ineq-gen} with $\gamma=3/4$, 
and this can easily be shown with the aid of Theorem \ref{QS-op-thm} and Lemma \ref{Dist-G-rep-lemma}.

\item We prove Theorem \ref{radial-thm-R4-gen} by using some of the techniques developed in \cite{Gon}. We transform inequality \eqref{michael-christ-ineq-gen}, using Laguerre polynomial expansions, into a series of finite-dimensional inequalities for doubly-stochastic matrices and we show that these matrices have spectral gaps uniformly bounded away from zero.
\end{enumerate}

\subsection{Conjectures for other even exponents}
In dimension $1$ there are two other even exponents $(p,q,d)=(6,6,1)$ and $(p,q,d)=(4,8,1)$ where analogous sharp inequalities should hold. Although sharing several similarities with the two-dimensional case, in the one-dimensional case the matrices that appear no longer are doubly stochastic (which is crucial in our proofs) and a new idea is needed to overcome this issue. However,  we have preformed numerical simulations that strongly suggest that that the following conjectures are true.

\begin{conjecture}
Let $g:\R\times \R \times \R\to \C$ be a function in $L^2(\R\times \R \times \R)$ such that $g(x,y,z)$ is even in each variable. Let $u(x,y,z,t)$ be a solution of {\rm (SE)} in $\R\times \R \times \R$ with initial data $g(x,y,z)$. Then there exists a universal constant $\alpha>0$ such that
\begin{align*}
&\int_{\R} \int_{\R} |u(x,x,x,t)|^2\dx \dt \\ & \leq \frac{1}{\sqrt{12}} \int_{\R\times \R\times \R} |g(x,y,z)|^2\dx\dy\dz - {\alpha} \, \Dist^2_{L^2(\R^3)}(g,L_{rad}^2(\R^3)).
\end{align*}
\end{conjecture}

\begin{conjecture}
Let $g:\R\times \R \times \R \times \R \to \C$ be a function in $L^2(\R\times \R \times \R \times \R)$ such that $g(x,y,z,w)$ is even in each variable. Let $u(x,y,z,w,t)$ be a solution of {\rm (SE)} in $\R\times \R \times \R \times \R$ with initial data $g(x,y,z,w)$. Then there exists a universal constant $\beta>0$ such that
\begin{align*}
& \int_{\R} \int_{\R\times \R} |u(x,x,y,y,t)|^2\dx\dy \dt \\ & \leq \frac{1}{4} \int_{\R\times \R\times \R \times \R} |g(x,y,z,w)|^2\dx\dy\dz\dw - {\beta} \, \Dist_{L^2(\R^4)}^2(g,L^2_{rad}(\R^4)).
\end{align*}
\end{conjecture}

Choosing $f\in L^2(\R)$ even and letting $g(x,y,z)=f(x)f(y)f(z)$ and $g(x,y,z,w)=f(x)f(y)f(z)f(w)$ respectively in the above conjectures, we would get sharpened Strichartz inequalities analogous to Theorem \ref{radial-thm-R4} for the exponents $(p,q,d)=(6,6,1)$ and $(p,q,d)=(4,8,1)$ respectively.
\section{Preliminaries}
In this section we present some preliminary results that will be used in the proof of Theorem \ref{radial-thm-R4-gen} in Section \ref{proof}.

\subsection{Inequalities for doubly stochastic matrices}
A matrix $A=[a_{i,j}]_{i,j=1,...,n}$ is said to be doubly stochastic if $a_{i,j}\geq 0$ for all $i,j$ and $A^t\1=A\1=\1$, where $\1=(1,1,...,1)$. In what follows $|\cdot|$ is the euclidean norm in $\C^n$, $\la\cdot,\cdot\ra$ is the Hermitian inner product in $\C^n$ and $\Dist(v,\la\1\ra)=\inf_{\lambda\in\C} \{|v-\lambda\1|\}$.

\begin{lemma}\label{matrix-ineq-lemma}
Let $A=[a_{i,j}]_{i,j=1,...,n}$ be doubly stochastic and assume that $\mu=n\min_{i,j} \{a_{i,j}\}>0$. Then for any vector $v\in\C^n$ we have
\lcalc{\label{matrix-ineq}
|\la Av, v\ra| \leq |v|^2 - \mu\Dist(v,\la\1\ra)^2.
}
\end{lemma}
\begin{proof}
Let $e_i$ denote the coordinate vectors. We have 
\calc{
|\la Av,e_i\ra|^2 &  = \bigg|\sum_{j=1}^n a_{i,j}v_j\bigg|^2 = \bigg|\sum_{j=1}^n \sqrt{a_{i,j}}\sqrt{a_{i,j}}v_j\bigg|^2 \\ 
& \leq\bigg( \sum_{j=1}^n a_{i,j}\bigg)\bigg(\sum_{j=1}^n{a_{i,j}}|v_j|^2\bigg) = \sum_{j=1}^n a_{i,j}|v_j|^2,
}
where we have used only the Cauchy-Schwarz inequality. Thus, we obtain
\calc{
|Av|^2 = \sum_{i=1}^n |\la Av,e_i\ra|^2 \leq \sum_{i=1}^n\sum_{j=1}^n a_{i,j}|v_j|^2 = \sum_{j=1}^n \bigg(\sum_{i=1}^na_{i,j}\bigg)|v_j|^2 = \sum_{j=1}^n |v_j|^2=|v|^2.
}
Now let $B=[\frac{a_{i,j}-\mu/n}{1-\mu}]_{i,j=1,...,n}$ (note that we always have $\mu\leq 1$ and $\mu=1$ if and only if $a_{i,j}=1/n$ for all $i,j$ and in this case inequality \eqref{matrix-ineq} is trivial). Clearly $B$ is also doubly stochastic and we obtain 
$$
|Bv|^2\leq |v|^2
$$
for all $v\in\C^n$. However, if $v$ is orthogonal to $\1$ then $(1-\mu)Bv = Av$. Let $v\in\C^n$ be any vector and write $v=c\1 + v_0$ where $v_0$ is orthogonal to $\1$. Note that $\Dist(v,\la\1\ra)=|v_0|$ and that $Av_0$ is also orthogonal to $\1$. We then obtain
\calc{
|\la Av, v\ra| = ||c|^2n + (1-\mu)\la Bv_0, v_0\ra| \leq |c|^2n + (1-\mu)|v_0|^2 = |v|^2 - \mu\Dist(v,\la\1\ra)^2.
}
This finishes the proof.
\end{proof}

We will also need another way of producing the same inequality of Lemma \ref{matrix-ineq-lemma} via spectral properties of $A$. Let 
$$
\sigma_1(A) = \sup \{|\sigma|: \sigma \text{ is an eigenvalue of } A\}
$$
and
$$
\sigma_2(A)=\sup \{|\sigma|:  \sigma \text{ is an eigenvalue of } A \text{ and } |\sigma|<\sigma_1(A) \}.
$$ 
Define the spectral gap of $A$ as follows
$$
\sg(A) = \sigma_1(A) - \sigma_2(A).
$$
If all the eigenvalues of $A$ have the same moduli define $\sg(A)=0$. Clearly, if $A$ is doubly stochastic then by Lemma \ref{matrix-ineq-lemma} we have $\sigma_1(A)=1$.

\begin{lemma}\label{matrix-lemma-symmetric}
Let $A=[a_{i,j}]_{i,j=1,...,n}$ be a doubly stochastic and symmetric matrix such that $\mu:=n\min_{i,j} \{a_{i,j}\}>0$. Then
$$
|\la Av, v\ra| \leq |v|^2 - \sg(A)\Dist(v,\la\1\ra)^2,
$$
for all $v\in\C^n$. Moreover:
\begin{enumerate}
\item $\sg(A)\geq \mu$;

\item Let $A^k = [a_{i,j}^{(k)}]_{i,j=1,...,n}$ denote the powers of $A$ and let $\theta\in[0,1]$. Then $\sg(A)\geq 1-\theta$ if and only if for some $C>0$ we have
$$
\sup_{i,j=1,...,n}|a_{i,j}^{(k)} - 1/n| \leq C\theta^k.
$$
\end{enumerate}
\end{lemma}
\begin{proof}
Let ${v_1,v_2,...,v_n}$ be an orthogonal basis of eigenvectors of $A$ with eigenvalues $\lambda_1\geq \lambda_2\geq ...\geq \lambda_n$ (repeated according their multiplicity), where $v_1=\1$ and $\sigma_1(A)=\lambda_1=1$. The assumption $\mu>0$ in conjunction with Lemma \ref{matrix-ineq-lemma} implies that the only eigenvalue of modulus one is $\lambda_1=1$, which in turn implies that $\sigma_2(A)=\max\{\lambda_2,-\lambda_n\}<1$. Let $v=c_1\1+c_2v_2+...+c_nv_n$. Noting that $\Dist(v,\la\1\ra)=|v-c_1v_1|$ we obtain
\calc{
|\la Av, v\ra| = \sum_{i=1}^n \lambda_i|c_i|^2|v_i|^2 & = \sum_{i=1}^n |c_i|^2|v_i|^2 - \sum_{i=2}^n (1-\lambda_i)|c_i|^2|v_i|^2 \\ &  \leq |v|^2 - \sg(A)\Dist(v,\la\1\ra)^2.
}

Item $(1)$ is a trivial consequence of the spectral gap being always non-negative. It is easy to see that $B=[\frac{a_{i,j}-\mu/n}{1-\mu}]_{i,j=1,...,n}$ is symmetric and doubly stochastic and that
$\sg(B)=(\sg(A)-\mu)/(1-\mu)\geq 0$ (again, if $\mu=1$ then $a_{i,j}=1/n$ for all $i,j$ and this lemma is trivial).

We now prove item $(2)$.  Letting $v=\la v,\1\ra\1/n+c_2v_2+...+c_nv_n$ we deduce that
\lcalc{\label{eq-Ak-v}
|A^{k}v - \la v,\1\ra\1/n|=| \lambda_2^kc_2v_2+...+\lambda_n^kc_nv_n| = O(\sigma_2(A)^k).
}
In particular,
$$
a_{i,j}^{(k)} - 1/n = \la A^k e_i,e_j \ra - \la e_i,\1 \ra\la \1,e_j \ra/n = O(\sigma_2(A)^k).
$$
Thus if $\sg(A)\geq 1-\theta$ then $\sigma_2(A)\leq \theta$ and we obtain  $|a_{i,j}^{(k)} - 1/n|\leq C\theta^k$ for some $C>0$. Conversely, assume that $|a_{i,j}^{(k)} - 1/n|\leq C\theta^k$. Taking $v=\1+v_1+v_n$ it is easy to see from  \eqref{eq-Ak-v} that 
$$
|A^kv - \1|\geq c\sigma_2(A)^k,
$$
for some $c>0$. However, we also have
\calc{
|A^kv - \1|^2 = |A^k(v_2+v_n)|^2 & = \sum_{i=1}^n \bigg|\sum_{j=1}^n a_{i,j}^{(k)}\la v_2+v_n,e_j\ra\bigg|^2 \\ & = \sum_{i=1}^n \bigg|\sum_{j=1}^n (a_{i,j}^{(k)}-1/n)\la v_2+v_n,e_j\ra\bigg|^2 \\ & \leq C^2n^2\theta^{2k}|v_2+v_n|^2.
}
Thus $\theta\geq \sigma_2(A)$, that is $\sg(A)\geq 1-\theta$.  This finishes the proof. 
\end{proof}

\subsection{Laguerre polynomials}
In what follows we will need some of the results presented in \cite[Section 2.2.1]{Gon} to perform our analysis and for that reason we follow most of the notation used there. 

For any $\nu>-1$ we denote by $\{L_n^{\nu}(x)\}_{n\geq 0}$ the generalized Laguerre polynomials associated with the parameter $\nu$ (we write $L_n(x)=L^0_n(x)$ for simplicity). In the sense of \cite[Chapters 2 and 5]{Sz}, these are the orthogonal polynomials associated with the measure $e^{-x}x^\nu \dx$ ($x>0$) and normalized by the condition
$$
\int_0^\infty |L_n^{\nu}(x)|^2 \frac{e^{-x}x^\nu \dx}{\Gamma(\nu+1)} = L_n^\nu(0)= \binom {n+\nu} n.
$$
They are known to form an orthogonal basis in the space $L^2(\R_+,e^{-x}x^\nu \dx)$ and, as a consequence, this implies that for any given dimension $d$ the functions
\begin{align*}
\Psi_n^\nu(x) = L_n^{\nu}(2\pi |x|^2) e^{-\pi|x|^2},
\end{align*}
with $\nu=d/2-1$, form an orthogonal basis in $L^2_{rad}(\R^d)$ and
$$
\|\Psi_n^\nu\|^2_{L^2(\R^d)}=2^{-(\nu+1)}\binom {n+\nu} n.
$$
We simply write $\Psi_n$ when $\nu=0$. This implies that the set 
$$
\{\Psi_m^\nu(x)\Psi_n^\nu(y)\}_{m,n\geq 0}
$$ 
forms an orthogonal basis in $L^2_{rad}(\R^d)\otimes L^2_{rad}(\R^d)$, that is, the sub-space of functions $g:\R^d \times \R^d\to \C$ in $L^2(\R^d\times \R^d)$ such that $g(x,y)$ is radial in $x$ and $y$. Thus, any function $g\in L^2_{rad}(\R^d)\otimes L^2_{rad}(\R^d)$ can be  uniquely written in the form
$$
g(x,y) = \sum_{m,n\geq 0} \p(m,n)\Psi^\nu_m(x)\Psi^\nu_n(x),
$$
for some coefficients $\p(m,n)$.

Let $\G=\ell^2(\Z_+^2)$ be the Hilbert space of complex sequences $\{\p(a,b)\}_{a,b\geq 0}$ with norm
$$
\|\p\|_\G^2:=\sum_{a,b\geq 0} |\p(a,b)|^2 < \infty.
$$
and Hermitian inner product
$$
\langle \p,\psi \rangle_\G = \sum_{a,b\geq 0} \p(a,b)\ov{\psi(a,b)}.
$$
Let $\Q:\G\to \G$ be the operator
\lcalc{\label{Qfull-opt}
\Q\p(a,b) = \sum_{\stackrel{c,d\geq 0}{c+d=a+b}} \p(c,d)Q(a,b,c,d),
}
where
\begin{align*}
Q(a,b,c,d)=\int_0^\infty L_a(x/2)L_b(x/2)L_c(x/2)L_d(x/2)e^{-x}\d x.
\end{align*}
For any integer $S\geq 0$, let $\G_S$ denote the subspace of sequences $\p:\Z^2_+\to \C$ such that $\p(a,b)=0$ if $a+b\neq S$. Clearly, the collection of spaces $\{\G_S\}_{S\geq 0}$ is orthogonal and their direct sum is dense in $\G$. We also have that $\dim(\G_S)=S+1$ and $\Q(\G_S)\subset \G_S$. Letting $\Q_S$ denote the restriction of $\Q$ to the subspace $\G_S$, we conclude that the operator $\Q_S$ can be represented by the following matrix 
\begin{align}\label{matrix-rep-Q}
\Q_S=[Q(a,S-a,c,S-c)]_{a,c=0,{{...}},S}.
\end{align}
It turns out that we can use the operator $\Q$ to identify the quantities appearing in Theorem \ref{radial-thm-R4-gen}.  The next theorem is implicit in the proof of \cite[Theorem 6]{Gon}, but it can easily be deduce from it and that is why we omit the proof.

\begin{theorem}\label{QS-op-thm}
Let $g(x,y)=\sum_{m,n\geq 0} \p(m,n)\Psi_n(x)\Psi_n(y)$ belong to $L^2(\R^2\times \R^2)$, where $g(x,y)$ is radial in $x\in \R^2$ and $y\in \R^2$. Let $u(x,y,t)$ solve {\rm (SE)} in $\R^2\times \R^2$ with initial data $g(x,y)$. We have
\begin{equation*}
\int_{\R} \int_{\R^2} |u(x,x,t)|^2\dx \dt = \frac{1}{16}\langle \p, \Q\p \rangle_\G
\end{equation*}
and
\begin{equation*}
\frac{1}{4} \int_{\R^2\times \R^2} |g(x,y)|\dx\dy = \frac{1}{16}\|\p\|_\G^2.
\end{equation*}
Moreover, for any $S\geq 0$ the matrix $\Q_S$ at \eqref{matrix-rep-Q} is a positive semi-definite doubly stochastic matrix with strictly positive entries. In particular, we conclude that $\|\Q\|_{\G\to\G}=1$. Furthermore, a sequence $\p\in \G$ satisfies 
$$
\langle \p, \Q\p \rangle_\G = \|\p\|_\G^2
$$
if and only if it has the property that $\p(a,b)=\p(c,d)$ whenever $a+b=c+d$.
\end{theorem}

Let $\G_{rad}$ denote the subspace of sequences $\p\in\G$ such that $\p(a,b)=\p(c,d)$ whenever $a+b=c+d$. In the same way as before, if $\Gamma\subset \G$ is a set we define
$$
\Dist(\p,\Gamma) = \inf \{\|\p-\psi\|_\G:\psi\in \Gamma\}.
$$
We have the following lemma.

\begin{lemma}\label{Dist-G-rep-lemma}
Let $g(x,y)=\sum_{m,n\geq 0} \p(m,n)\Psi_m(x)\Psi_n(y)$  belong to $L^2_{rad}(\R^2) \otimes L^2_{rad}(\R^2)$. Then
\calc{
\Dist\left(g,L^2_{rad}(\R^4)\right)^2 =\frac{1}{4}\Dist(\p,\G_{rad})^2.
}
\end{lemma}
\begin{proof}
Since $$\left\{\frac{\Psi^1_S}{\sqrt{(S+1)/4}}\right\}_{S\geq 0}$$ is an orthonormal basis of $L^2_{rad}(\R^4)$ we have
\calc{
\Dist\left(g,L^2_{rad}(\R^4)\right)^2 = \left\|g-  P_{rad}(g) \right\|_{L^2(\R^4)}^2,
}
where 
$$
P_{rad}(g) = \sum_{S\geq 0} \frac{\la g, \Psi^1_S\ra_{L^2(\R^4)}}{(S+1)/4} {\Psi^1_S}
$$
is the projection of $g$ in the space $L^2_{rad}(\R^4)$. An important formula related to Laguerre polynomials is the summation formula \eqref{add-formula}, which implies that
\calc{
\Psi_S^1 = \sum_{a+b=S} \Psi_a \otimes \Psi_b.
}
Using the above formula we obtain
\calc{
\la g, \Psi^1_S\ra_{L^2(\R^4)} & = \sum_{n,m\geq 0} \sum_{a+b=S} \p(m,n)\int_{\R^2} \Psi_n(x)\Psi_a(x)\dx \int_{\R^2}\Psi_m(y)\Psi_b(y)\dy \\
& = \sum_{a+b=S} \p(a,b)\int_{\R^2} \Psi_a(x)^2\dx\int_{\R^2} \Psi_b(y)^2\dy \\
& = \frac{1}{4}  \sum_{a+b=S} \p(a,b).
}
We conclude that
$$
P_{rad}(g) =  \sum_{S\geq 0}\left(\frac{\sum_{a+b=S} \p(a,b)}{S+1}\right) \Psi^1_S.
$$
This implies that
\calc{
\Dist^2&\left(g,  L^2_{rad}(\R^4)\right) \\ &= \left\|\sum_{n,m\geq 0} \p(m,n)\Psi_n\otimes\Psi_m -  \sum_{S\geq 0}\left(\frac{\sum_{a+b=S} \p(a,b)}{S+1}\right) \Psi^1_S \right\|_{L^2(\R^4)}^2 \\
& =  \left\|\sum_{S\geq 0} \sum_{n+m=S}\left[ \left(\p(m,n) - \frac{\sum_{a+b=S} \p(a,b)}{S+1}\right)\Psi_n\otimes\Psi_m\right]  \right\|_{L^2(\R^4)}^2 \\
& = \frac{1}{4}  \sum_{S\geq 1} \sum_{m+n=S} \bigg|\p(m,n) -  \frac{\sum_{a+b=S} \p(a,b)}{S+1}\bigg|^2
}

On the other hand, let $P_S:\G\to\G_S$ denote the projection onto the space $\G_S$. Let  $\1_S:\Z^2_+\to\C$ be defined as $\1_S(a,b)=1$ if $a+b=S$ and $\1_S(a,b)=0$ if $a+b\neq S$. Letting $\Dist(\p,\la\1_S\ra)=\inf_{\lambda\in\C}\{\|\p-\lambda\1_S\|_\G\}$ we obtain
\lcalc{\label{dist_G-crucial-rel}
 \Dist^2(\p,\G_{rad}) & = \sum_{S\geq 1} \Dist^2(P_S(\p),\la\1_S\ra) \\ & = \sum_{S\geq 1} \sum_{a+b=S} \bigg|\p(a,b) -  \frac{\sum_{c+d=S} \p(c,d)}{S+1}\bigg|^2.
}
This finishes the proof.
\end{proof}

Theorem \ref{QS-op-thm} and Lemma \ref{Dist-G-rep-lemma} imply that Theorem \ref{radial-thm-R4-gen} is equivalent to the inequality
\lcalc{\label{ineq-in-G}
\langle \p, \Q\p \rangle_\G \leq \|\p\|_\G^2 - \gamma\Dist(\p,\G_{rad})^2
}
for all $\p\in\G$ (the constant $\gamma$ above being the same as in inequality \eqref{michael-christ-ineq-gen}). By Lemma \ref{matrix-lemma-symmetric} we have
\calc{
\langle \p, \Q\p \rangle_\G \leq \|\p\|_\G^2 - \sg(Q_S) \Dist(\p,\la\1_S\ra)^2,
}
for all $\p\in\G_S$. Using identity \eqref{dist_G-crucial-rel} in conjunction with the fact that the spaces $\G_S$ decompose $\G$ into a sum of mutually orthogonal subspaces  we obtain
\calc{
\langle \p, \Q\p \rangle_\G & = \sum_{S\geq 0} \langle P_S(\p), \Q P_S(\p) \rangle_\G \\ &  \leq \sum_{S\geq 0} \|P_S(\p)\|_\G^2 - \sum_{S\geq 1}\sg(\Q_S) \Dist(P_S(\p),\la\1_S\ra)^2 \\
& \leq \sum_{S\geq 0} \|P_S(\p)\|_\G^2 - \inf_{S\geq 1}\{\sg(\Q_S)\}\sum_{S\geq 1} \Dist(P_S(\p),\la\1_S\ra)^2 \\
& = \|\p\|_\G^2 - \inf_{S\geq 1}\{\sg(\Q_S)\}\Dist(\p,\G_{rad}).
}
Hence, if $\inf_{S\geq 1}\{\sg(\Q_S)\}>0$ then inequality \eqref{ineq-in-G} holds with 
$$
\gamma=\inf_{S\geq 1}\{\sg(\Q_S)\}.
$$
Thus if we prove that the sequence of matrices $\Q_S$ have spectral gaps uniformly bounded away from zero we prove Theorem \ref{radial-thm-R4-gen}. We compile this information in the following Lemma.

\begin{lemma}\label{spectral-gap-lemma}
If 
$$
\inf_{S\geq 1}\left\lbrace\sg(\Q_S)\right\rbrace \geq \delta >0
$$
then Theorem \ref{radial-thm-R4-gen} holds with
$$
\gamma={\delta}.
$$
\end{lemma}

\section{Proof of Theorem \ref{radial-thm-R4-gen}}\label{proof}
Consider the matrix $\Q_S$ defined in \eqref{matrix-rep-Q}. We will show that the spectral gap of $\Q_S$ is uniformly bounded from below. In particular, we will show that 
\lcalc{\label{Q-spectral-ineq}
\sg(\Q_S) \geq \frac{4}{\pi^2}
}
for all $S\geq 1$. Hence, by the Lemma \ref{spectral-gap-lemma} we conclude that Theorem \ref{radial-thm-R4-gen} is true with $\gamma=\frac{4}{\pi^2}$. We note that the above lower bound is not best possible and numerical computations show that $\sg(\Q_1)=1$ and suggest that
$$
\sg(\Q_S)=3/4
$$
for all $S\geq 2$. We address this issue in the Section \ref{numerics}.

\noindent {\bf Step 1.} Let $\Q_S^\ell=[Q^{(\ell)}(a,S-a,c,S-c)]_{a,c=0,...,S}$ denote the powers of matrix $\Q_S$. For simplicity we write $L_{a_1a_2...a_n}(x)=L_{a_1}(x)L_{a_2}(x)...L_{a_n}(x)$ for given integers $a_1,a_2,...,a_n$. We then deduce that
\calc{
Q^{(2)}(a,b,&c,d)\\  & = \int_0^\infty L_{ab}(x/2)\int_0^\infty \bigg[\sum_{m+n=S}L_{nm}(x/2)L_{mn}(y/2)\bigg]L_{cd}(y/2)e^{-y}\dy e^{-x}\dx
}
if $a+b=c+d=S$. Let
$$
K_S(x,y)=\sum_{m+n=S}L_{nm}(x/2)L_{mn}(y/2)
$$
and define the following kernel operator over $L^2([0,\infty),e^{-x}\dx)$
\lcalc{\label{KS-def}
\K_S(f)(x) = \int_0^\infty K_S(x,y)f(y)e^{-y}\dy.
}
We conclude that
$$
Q^{(2)}(a,b,c,d) = \int_0^\infty L_{ab}(x/2) \K_S(L_{cd}(\cdot/2))(x)e^{-x}\dx
$$
if $a+b=c+d=S$. It is now a straightforward calculation to deduce that
\lcalc{\label{Q-powers-rep}
Q^{(\ell+1)}(a,b,c,d) = \int_0^\infty L_{ab}(x/2) \K_S^{(\ell)}(L_{cd}(\cdot/2))(x)e^{-x}
}
for any integer $\ell \geq 1$, where $\K_S^{(\ell)}$ is the $\ell$-fold composition of $\K_S$.

\noindent {\bf Step 2.} We are going to perform a spectral analysis on $\K_S$ which ultimately will give us spectral information about matrix $\Q_S$. In this direction, we will need to represent $\K_S$ in the basis $\{L_n(x)\}_{n\geq 0}$. We have the following lemma that we postpone the proof for the final steps.

\begin{lemma}\label{crucil-rep-lemma}
If $n>S$ then 
\lcalc{\label{KS-zeros}
\K_S(L_n) = 0.
}
If $0\leq m,n\leq S$ we have
\lcalc{\label{coeff-K-S}
\int_0^\infty L_m(x)\K_S(L_n)(x)e^{-x}\dx =\frac{1}{4^S} \sum {\binom {2i} i}{\binom {2j} j}{\binom {2u} u}{\binom {2v} v},
}
where the summation is taken over $i,j,u,v\geq 0$ such that $i+j=S-n$, $u+v=n$ and $j+v=m$. Moreover, if $\kappa_{m,n}$ is the quantity on the left hand side of \eqref{coeff-K-S} then the matrix $ [\kappa_{m,n}]_{m,n=0,...,S}$ is symmetric and doubly stochastic.
\end{lemma}
By Lemma \ref{crucil-rep-lemma} the operator $\K_S$ has a finite dimensional range and can be represented by the following matrix (with an abuse of notation)
$$
\K_S = [\kappa_{m,n}]_{m,n=0,...,S}
$$
where
$$
\kappa_{m,n} = \int_0^\infty L_m(x)\K_S(L_n)(x)e^{-x}\dx.
$$
Recall that $\int_0^\infty L_n(x)^2e^{-x}\dx=1$, hence, roughly speaking, each $L_n(x)$ works as the coordinate vector $e_n$. We claim that
\calc{
\kappa_{m,n}\geq \frac{4}{\pi^2(S+1)}.
}
First, we have the following inequality (which can be derived from Stirling's formula)
\lcalc{\label{stirlin-ineq}
\binom {2p} p \frac{1}{4^p}\geq \frac{1}{\sqrt{\pi(p+1/2)}}
}
for all $p\geq 0$ and the quotient between both sides above converge to $1$ as $p\to \infty$. Using this inequality in conjunction with identity \eqref{coeff-K-S} we obtain
\lcalc{\label{ineq-0}
\pi^2\kappa_{m,n}\geq \sum \frac{1}{\sqrt{(i+1/2)(j+1/2)(u+1/2)(v+1/2)}}
}
where the summation is taken over $i,j,u,v\geq 0$ such that $i+j=S-n$, $u+v=n$ and $j+v=m$. Note that these conditions imply that $i+u=S-m$. Secondly, using inequality $1/\sqrt{ts}\geq 2/(t+s)$ for $t,s>0$, we  obtain
\lcalc{\label{ineq-1}
\frac{1}{\sqrt{(i+1/2)(j+1/2)(u+1/2)(v+1/2)}} &\geq \frac{4}{(i+u+1)(j+v+1)} \\ & = \frac{4}{(S-m+1)(m+1)},
}
and
\lcalc{\label{ineq-2}
\frac{1}{\sqrt{(i+1/2)(j+1/2)(u+1/2)(v+1/2)}} &\geq \frac{4}{(i+j+1)(u+v+1)} \\ & = \frac{4}{(S-n+1)(n+1)}.
}
We separate our argument in cases. If $m\leq  \min\{n,S-n\}$ then by  \eqref{ineq-0} and \eqref{ineq-1} we have
$$
\pi^2\kappa_{m,n}\geq  \frac{4}{(S-m+1)(m+1)} \sum_{j=0}^m1 = \frac{4}{S-m+1} \geq \frac{4}{S+1}.
$$
If $m>  \max\{n,S-n\}$ then by  \eqref{ineq-0} and \eqref{ineq-1} we have
$$
\pi^2\kappa_{m,n}\geq  \frac{4}{(S-m+1)(m+1)} \sum_{j=m-n}^{S-n}1 = \frac{4}{m+1} \geq \frac{4}{S+1}.
$$
If $S-n<m\leq n$ then by  \eqref{ineq-0} and \eqref{ineq-2} we have
$$
\pi^2\kappa_{m,n}\geq  \frac{4}{(S-n+1)(n+1)} \sum_{j=0}^{S-n}1 = \frac{4}{n+1} \geq \frac{4}{S+1}.
$$
If $n<m\leq S-n$ then by  \eqref{ineq-0} and \eqref{ineq-2} we have
$$
\pi^2\kappa_{m,n}\geq  \frac{4}{(S-n+1)(n+1)} \sum_{j=m-n}^{m}1 = \frac{4}{S-n+1} \geq \frac{4}{S+1}.
$$
This proves the claim.

We can now apply Lemma \ref{matrix-lemma-symmetric} item $(1)$ to extract information about the spectral gap of $\K_S$, that is, 
$$
\sg(\K_S)\geq \frac{4}{\pi^2}.
$$
By  Lemma \ref{matrix-lemma-symmetric} item $(2)$ we deduce that
\lcalc{\label{kappa-S+1-ineq}
|\kappa_{m,n}^{(\ell)}-1/(S+1)| \leq C\left(\frac{\pi^2-4}{\pi^2}\right)^{\ell}
}
for some constant $C$ which does not depend on $\ell$, where 
$$
\kappa_{m,n}^{(\ell)} = \int_0^\infty L_m(x)\K_S^{(\ell)}(L_n)(x)e^{-x}\dx
$$
are the coefficients of the $\ell$ power of matrix $\K_S$ associated with the $\ell$-fold composition $\K_S^{(\ell)}$.

\noindent {\bf Step 3.}  An important formula for Laguerre polynomials is the summation formula \cite[Formula 8.977-1]{GR}, which has the following identity as a particular case
\lcalc{\label{add-formula}
L^1_S(x + y) =  \sum_{n=0}^S L_n(x)L_{S-n}(y).
}
Plugging $y=0$ we also have $L^1_S(x) =  \sum_{n=0}^S L_n(x)$ . Now, for any $a,b\geq 0$ with $a+b=S$ let 
$$
L_{ab}(x/2) = \sum_{m=0}^S p_m(a,b)L_n(x),
$$
where $p_n(a,b)$ are the coefficients of the expansion of $L_{ab}(x/2)$ in terms of the Laguerre polynomials $L_n(x)$. We obtain that
\calc{
\sum_{m=0}^S p_m(a,b)=L_{ab}(0)=1.
}
We can now go back to the matrix $\Q_S^{\ell}$ and use identity \eqref{Q-powers-rep} to deduce that
\calc{
 \Q_S^{\ell+1}(a,b,c,d) -\frac{1}{S+1} & = \int_0^\infty L_{ab}(x/2)\K_S^{(\ell)}[L_{cd}(\cdot/2)](x)e^{-x}\dx -\frac{1}{S+1} \\ &= \sum_{n,m=0}^S p_m(a,b)p_n(c,d) \kappa_{m,n}^{(\ell)}  -\frac{1}{S+1}\\
 & = \sum_{n,m=0}^S p_m(a,b)p_n(c,d) \left[\kappa_{m,n}^{(\ell)}  -\frac{1}{S+1}\right]
}
if $a+b=c+d=S$. We can now apply inequality \eqref{kappa-S+1-ineq} to obtain that
$$
\left| Q_S^{\ell+1}(a,b,c,d) -\frac{1}{S+1}\right| \leq  C\left(\frac{\pi^2-4}{\pi^2}\right)^{\ell}
$$
for some constant $C$ independent of $\ell$. Finally, we can apply again Lemma \ref{matrix-lemma-symmetric} item $(2)$ to deduce that 
$$
\sg(\Q_S)\geq \frac{4}{\pi^2}.
$$
This proves the desired inequality \eqref{Q-spectral-ineq} and finishes the proof of the theorem.

\noindent {\bf Step 4.} We now turn our attention to the proof of Lemma \ref{crucil-rep-lemma}. The Poisson kernel associated with the Laguerre polynomials ${L_n(x)}$ is given by (see \cite[Formula 8.976-1]{GR})
\lcalc{\label{Poisson-kernel-Lag}
P(x,y;w) = \sum_{n\geq 0} L_n(x)L_n(y)w^n = \frac{e^{-\frac{(x+y)w}{1-w}}}{1-w}I_0\left(2\frac{\sqrt{wxy}}{1-w}\right)
}
and defined for all $0<w<1$, where the function $I_0$ above is the modified Bessel function of the first kind associated with parameter $\nu=0$ (see \cite[Section 8.4]{GR} for basic facts about Bessel functions). This is an important kernel since realizes the multiplication operator $L_n\mapsto w^n L_n$, that is,
$$
\int_0^\infty P(x,y;w)L_n(y)e^{-y}\dy = w^nL_n(x).
$$
Recall now the definition \eqref{KS-def} of $\K_S$ and its associated kernel $K_S(x,y)$. For $0<w<1$ we obtain
\calc{
G(x,y;w) := \sum_{S\geq 0} K_S(x,y)w^S & = \sum_{S\geq 0}\sum_{m+n=S} L_{m}(x/2)L_m(y/2)w^S \\
& =P(x/2,y/2;w)^2 \\
& = \frac{e^{-\frac{(x+y)w}{1-w}}}{(1-w)^2}I_0\left(\frac{\sqrt{wxy}}{1-w}\right)^2.
}
Given its form above, it is reasonable to expect that the kernel $G(x,y;w)$ behaves similarly to $P(x,y;w)$ and that is indeed the case. Using identity \cite[p.150(1)]{Wat} and the fact that $I_0(z)=J_0(iz)$ ($J_0(z)$ is the Bessel function of the first kind) we derive the following identity
$$
I_0(z)^2 = \fint_0^{\pi/2} I_0(2z\cos(t))\dt,
$$
where the integral sign above stands for $\frac{2}{\pi}\int_0^{\pi/2}$. We then obtain the following useful representation linking the kernels $G(x,y;w)$ and $P(x,y;w)$
\calc{
G(x,y;w)  &=  \fint_0^{\pi/2} \frac{e^{-\frac{(x+y)w}{1-w}}}{(1-w)^2} I_0\left(2\cos(t)\frac{\sqrt{wxy}}{1-w}\right)\dt \\
& =\fint_0^{\pi/2} \frac{e^{-\frac{xw\sin^2(t)}{1-w}}}{1-w} \frac{e^{-\frac{([x\cos^2(t)]+y)w}{1-w}}}{(1-w)} I_0\left(2\frac{\sqrt{w[x\cos^2(t)]y}}{1-w}\right)\dt \\
& = \fint_0^{\pi/2}\frac{e^{-\frac{xw\sin^2(t)}{1-w}}}{1-w} P(x\cos^2(t),y;w)\dt.
}
This implies that
\calc{
\sum_{S\geq 0} \K_S(L_n)(x) w^S & = \int_{0}^\infty G(x,y;w)L_n(y)e^{-y}\dy \\
& = \fint_0^{\pi/2}\frac{e^{-\frac{xw\sin^2(t)}{1-w}}}{1-w} \int_0^\infty P(x\cos^2(t),y;w)L_n(y)e^{-y}\dy\dt \\
& = w^n\fint_0^{\pi/2}\frac{e^{-\frac{xw\sin^2(t)}{1-w}}}{1-w}L_n(x\cos^2(t))\dt.
}
Using now the following generating function for the Laguerre polynomials $\{L_n(x)\}$ (which can be derived from \eqref{Poisson-kernel-Lag} by plugging $y=0$)
\lcalc{\label{gen-func-lag}
\sum_{\ell\geq 0} L_\ell(x)w^\ell = \frac{e^{-\frac{xw}{1-w}}}{1-w},
}
we finally obtain
\calc{
\sum_{S\geq 0} \K_S(L_n)(x) w^S & = \sum_{\ell\geq n} w^{\ell} \fint_0^{\pi/2} L_{\ell-n}(x\sin^2(t))L_n(x\cos^2(t))\dt.
}
Comparing the powers of $w$ in the above expression we conclude that
$$
\K_S(L_n)=0
$$
if $n>S$, which already shows identity \eqref{KS-zeros} of Lemma \ref{crucil-rep-lemma}, and that
\lcalc{\label{KS-int-id}
\K_S(L_n)(x) = \fint_0^{\pi/2} L_{S-n}(x\sin^2(t))L_n(x\cos^2(t))\dt.
}
if $0\leq n\leq S$.

\noindent {\bf Step 5.} We will now use identity \eqref{KS-int-id} to prove identity \eqref{coeff-K-S} of Lemma \ref{crucil-rep-lemma} and finish the proof. Define the following function
\lcalc{\label{T-id1}
T(w_1,w_2,w_3) = \sum_{a,b,c\geq 0} \fint_0^{\pi/2} \int_{0}^\infty L_{a}(x\sin^2(t))L_b(x\cos^2(t))L_c(x)e^{-x}\dx\dt \, w_1^aw_2^bw_3^c
}
for $0<w_1,w_2,w_3<1$. Using formula \eqref{gen-func-lag} we obtain that
\calc{
T&(w_1,w_2, w_3) \\ & = 
\frac{1}{(1-w_1)(1-w_2)(1-w_3)}\fint_0^{\pi/2} \int_0^\infty e^{-x\left[1+\frac{\sin^2(t)w_1}{1-w_2}+ \frac{\cos^2(t)w_2}{1-w_2}  + \frac{w_3}{1-w_3}\right]}\dx\dt \\ & =
\fint_0^{\pi/2} \left[1-(w_1+w_2w_3)\cos^2(t) - (w_2+w_1w_3)\sin^2(t)+w_1w_2w_3\right]^{-1}\dt \\
& = [1+w_1w_2w_3]^{-1}  \fint_0^{\pi/2}  \left[1-A\cos^2(t) - B\sin^2(t)\right]^{-1}\dt
}
where $A=\frac{w_1+w_2w_3}{1+w_1w_2w_3}$ and $B=\frac{w_2+w_1w_3}{1+w_1w_2w_3}$. The final integral above is doable via  a change the change of variables $s=\sqrt{\frac{1-A}{1-B}}\tan(t)$ and we obtain
\calc{
T(w_1,w_2, w_3) & =  [1+w_1w_2w_3]^{-1} [(1-A)(1-B)]^{-1/2} \\
& = [(1-w_1)(1-w_2w_3)(1-w_2)(1-w_1w_3)]^{-1/2} \\
& = \sum_{i,j,u,v\geq 0} \binom {2i} i \binom {2j} j \binom {2u} u \binom {2v} v \frac{w_1^i(w_2w_3)^vw_2^u(w_1w_3)^j}{4^{(i+j+u+v)}}  \\
& = \sum_{a,b,c\geq 0} \frac{w_1^aw_2^bw_3^c}{4^{a+b}} \sum_{\substack{i+j=a \\ u+v=b \\ j+v=c \\ i,j,u,v\geq 0}}\binom {2i} i \binom {2j} j \binom {2u} u \binom {2v} v,
}
where above we used the power series expansion $[1-w]^{-1/2} = \sum_{n\geq 0} 4^{-n}\binom {2n} n w^n$. Using \eqref{KS-int-id} and comparing the power series coefficients of $T$ in the above identity with definition \eqref{T-id1}, we conclude that identity \eqref{coeff-K-S} of Lemma \ref{crucil-rep-lemma} is true.

It remains to show that $[\kappa_{n,m}]_{n,m=0,...,S}$ is symmetric and doubly stochastic. The fact that $\K_S$ is self-adjoint (it is given by a real-valued kernel) clearly implies that the matrix $[\kappa_{n,m}]_{n,m=0,...,S}$ is symmetric. Using \eqref{KS-int-id} and \eqref{add-formula} we obtain
\calc{
\K_S(L_S^1)(x) & = \sum_{n=0}^S \K_S(L_n)(x) =  \sum_{n=0}^S \fint_0^{\pi/2} L_{S-n}(x\sin^2(t))L_n(x\cos^2(t))\dt\\
& = \fint_0^{\pi/2} L^1_{S}(x)\dt = L^1_{S}(x).
}
That is, $L_S^1(x)$ is the eigenfunction associated with the eigenvalue $1$. This implies that
\calc{
\sum_{n=0}^S \kappa_{m,n} & = \int_0^\infty L_m(x)\left[\sum_{n=0}^S \K_S(L_n)(x)\right]e^{-x}\dx =  \int_0^\infty L_m(x)L^1_{S}(x)e^{-x}\dx \\
& =  \int_0^\infty L^1_{S}(x)e^{-x}\dx +  \int_0^\infty [L_m(x)-1]L^1_{S}(x)e^{-x}\dx \\ &= 
 \sum_{n=0}^S \int_0^\infty L_n(x)e^{-x}\dx +  \int_0^\infty \frac{L_m(x)-1}{x}L^1_{S}(x)xe^{-x}\dx \\
& = 1 + 0,
}
where above we used that $L_S^1(x)$ is orthogonal, with respect to $xe^{-x}\dx$, to any polynomial with degree less than $S$ (recall that $L_m(0)=1$). This proves that $[\kappa_{n,m}]_{n,m=0,...,S}$ is also doubly stochastic and finishes the proof of the Lemma \ref{crucil-rep-lemma}.


\section{Concluding Remarks}\label{numerics}
\subsection{A combinatorial point of view}
``{\it Members of four different clubs, each wearing a hat with an insignia of his club, hang
their hats on entering the hall. When they leave there is a power failure and the departing
guests scramble for hats in the dark. Assuming the hats were picked at an entirely random
fashion, would you bet that the number of guests wearing hats with wrong insignias is
even}?''

This is a very nice extract from \cite{GZ}, where the authors continue the work initiated in \cite{AIK} and give a pure combinatorial proof of the following remarkable fact:
\calc{
&2^{a+b+c+d}\int_0^\infty L_{a}(x/2)L_{b}(x/2)L_{c}(x/2)L_{d}(x/2)e^{-x}\dx \\ & = \#\big\{\text{Events where we have an even number of guests with wrong hats}\big\} \\ & \ \ \ \  - \#\big\{\text{Events where we have an odd number of guests with wrong hats} \big\}  \\ & > 0,
}
where $a,b,c,d$ are respectively the number of members in each club. We conclude that is more likely to have an even number of guests wearing hats with wrong insignias. What is also a remarkable coincidence is that these same coefficients appear in the calculation of the $\|u\|_{L^4(\R^2\times \R)}$-norm for a solution of the Schr\"odinger equation $u(x,t)$ in two dimensions and that information about these coefficients can be translated into information about $u(x,t)$ (some more details in \cite[Appendix]{Gon}). Moreover, Conjecture \ref{Q-min-conj} points to a not yet known quantitative lower bound
\begin{align*}
& \big\{\text{Events where we have an even number of guests with wrong hats}\big\} \\ & \ \ \ \  - \#\big\{\text{Events where we have an odd number of guests with wrong hats} \big\} \\ & \geq \frac{2^{a+b+c+d+1}}{\pi(a+b+c+d+1)},
\end{align*}
if $a+b=c+d$, which is best possible asymptotically (besides the multiplying constant) if $d=0$ and $|a-b|\leq 1$.

\subsection{Numerical simulations} In this part we comment about numerical simulations done with the help of MATLAB \cite{matlab} and PARI/GP \cite{GP} and the conjectures they seem to indicate. 

First, one can simply plot representations of matrices $\Q_S$ in shades of gray, where larger entries of $\Q_S$ produce darker tones.  By inspection we find out that larger values accumulate at the diagonals of $\Q_S$ and smaller values in the mid rows and columns. This pattern is repeated in every single representation of $\Q_S$ we were able to compute and they directly point to the following conjecture, which we verified to hold for $S\leq 30$. 


\begin{conjecture}\label{Q-min-conj}
Let $a+b=c+d=S$. Then
$$
Q(a,b,c,d) \geq Q(\lfloor S/2 \rfloor, \lceil S/2 \rceil, S, 0).
$$
That is, the minimal element of $\Q_S$ lies in the first column with the middle row.
\end{conjecture}
\noindent It is a fun calculation (that we leave to the reader) using the generating function \eqref{gen-func-lag} that we have
$$
Q(a,b,a+b,0) = \frac{\binom {2a} a \binom {2b} b}{4^{a+b}} \geq \frac{2}{\pi(a+b+1)}.
$$
Thus, Conjecture \ref{Q-min-conj} in conjunction with Lemmas \ref{matrix-lemma-symmetric} and \ref{spectral-gap-lemma} would imply Theorem \ref{radial-thm-R4-gen} with $\gamma = 2/\pi$ (hence producing a better constant than $\gamma=4/\pi^2$). 

Secondly, one can try to compute eigenvalues. Numerical calculations of the eigenvalues of $\Q_S$  suggest the existence of a very structured relation between these matrices for different $S$'s. Let $Eig(\Q_S)$ denote the set of eigenvalues of $\Q_S$ and let $Eig(\Q)$ be the set of eigenvalues of the full operator $\Q$ defined in \eqref{Qfull-opt}. It is easy to see that $\cup_{S\geq 0} Eig(\Q_S)=Eig(\Q)$. However, numerical simulations point to the following conjecture.

\begin{conjecture}\label{Q-eigs}
Let $\lambda(n)=\binom {2n} n^2 16^{-n}$. Then
$$
Eig(\Q_S) = \{\lambda(0), \lambda(1),..., \lambda\left(\left\lfloor S/2\right\rfloor\right), 0\}.
$$
Moreover, each non-zero eigenvalue has multiplicity $1$ and the zero eigenvalue has multiplicity $\left\lceil S/2\right\rceil$.
\end{conjecture}

The Laguerre polynomials expand in monomials with rational coefficients whenever the parameter $\nu$ is rational, therefore using rational arithmetic one can compute the coefficients $Q(a,b,c,d)$ explicitly, which will consist of rational numbers. Thus, we can compute the characteristic polynomial $p_S(\lambda)$ of each $\Q_S$, which will then have only rational coefficients as well. Thus, we can precisely evaluate $p_S(\lambda(n))$ using rational arithmetic and verify that it vanishes at each $\lambda(0), \lambda(1),..., \lambda\left(\left\lfloor S/2\right\rfloor\right)$ with order $1$ and vanishes at $\lambda=0$ with order $\left\lceil S/2\right\rceil$. Using this procedure we confirmed the conjecture above for $S\leq 30$.   

One way of guessing this conjecture is by plotting the eigenvalues of, say, $\Q_{30}$ and realize that they decrease as $1/n$. Then plotting the difference of the reciprocals of the eigenvalues of $\Q_{30}$ we can clearly see they approximating $\pi$. This suggests that they have the following asymptotic approximation $1/(\pi n)$. However, using \eqref{stirlin-ineq} we can also try the approximation $\lambda(n)$ (since $\lambda(n)\sim 1/(\pi n)$). It turns out that this was so remarkably accurate that it could only be case that $\lambda(n)$ is the true value for these eigenvalues. In particular, Conjecture \ref{Q-eigs} would imply that $\sg(\Q_S)=\lambda(0)-\lambda(1)=3/4$ for all $S\geq 2$, and we would be able to use Lemma \ref{spectral-gap-lemma} to prove Theorem \ref{radial-thm-R4-gen} with $\gamma=3/4$, which is best possible.

\section*{Acknowledgments}
We acknowledge the support from the the StartUp funds from the Faculty of Sciences of the University of Alberta.




\begin{thebibliography}{99}


\bibitem{AIK}
R. Askey, M. Ismail, and T. Koorwinder,
\newblock  Weighted permutation problems and Laguerre polynomials,
\newblock Journal of Combinatorial Theory, Series A, 25(3)(1978), 277-287


\bibitem{Be}
W. Beckner, 
\newblock Inequalities in Fourier Analysis,
\newblock Annals of Mathematics, 102 (1975), 159-182.

\bibitem{BS}
N. Bez and M. Sugimoto, 
\newblock Optimal constants and extremizers for some smoothing
estimates, 
\newblock Journal d'Analyse Math\'ematique, vol. 131, issue 1, (2017), p. 159-187


\bibitem{BBCH}
J. Bennett, N. Bez, A. Carbery and D. Hundertmark,
\newblock Heat-Flow monotonicity of Strichartz norms,
\newblock Analysis and Partial Differential Equations 2(2)(2008), 147-158.

\bibitem{Ca}
E. Carneiro,
\newblock A sharp inequality for the Strichartz norm,
\newblock  Int. Math. Res. Notices 2009 (2009), 3127-3145.

\bibitem{CO}
E. Carneiro and D. Oliveira e Silva,
\newblock Some Sharp Restriction Inequalities on the Sphere,
\newblock Int. Math. Res. Not., vol 2015, issue 17, p. 8233-8267.

\bibitem{Jesus}
M. Christ,
\newblock A sharpened Hausdorff-Young inequality,
\newblock  	arXiv:1406.1210.

\bibitem{MC}
M. Christ,
\newblock On Nearly Radial Product Functions,
\newblock arXiv:1506.00155.

\bibitem{F}
D. Foschi,
\newblock Global maximizers for the sphere adjoint Fourier restriction inequality,
\newblock Journal of Functional Analysis, vol. 268, issue 3, (2015), p. 690-702.

\bibitem{Fo} 
D. Foschi, 
\newblock Maximizers for the Strichartz inequality, 
\newblock J. Eur. Math. Soc. 9 (2007), 739-774.

\bibitem{Gon}
F. Gon\c{c}alves,
\newblock Orthogonal polynomials and sharp estimates for the Schr\"odinger Equation,
\newblock Int. Math. Res. Not., vol. 2017, no. 00, p. 1-28.

\bibitem{GZ}
J. Gillis and D. Zeilberger,
\newblock A direct combinatorial proof of a positivity result,
\newblock Europ. J. Combinatorics 4 (1983), 221-223.

\bibitem{GR}
I. S. Gradshteyn and I. M. Ryzhik,
\emph{Table of integrals, series, and products.}
Translated from the Russian. Seventh edition. Elsevier/Academic Press, Amsterdam, (2007).
 
\bibitem{HZ} 
D. Hundertmark and V. Zharnitsky, 
\newblock On sharp Strichartz inequalities in low dimensions, 
\newblock Int. Math. Res. Not. (2006), 1-18.

\bibitem{matlab}
\newblock MATLAB 2017a, The MathWorks, Inc., Natick, Massachusetts, United States.

\bibitem{GP}
C. Batut, K. Belabas, D. Bernardi, H. Cohen and M. Olivier,
\newblock User's Guide to PARI-GP,
\newblock Laboratoire A2X, Universit\'e Bordeaux I, France, 1998.

\bibitem{Sz}
G. Szeg\"o,
\newblock {\it Orthogonal polynomials},
\newblock American Mathematical Society Colloquium Publications Volume XXIII, Fourth Edition, 1975.

\bibitem{Tao}
T. Tao,
\newblock {\it Nonlinear dispersive equations: Local and global analysis}, 
\newblock CBMS Regional Conference Series in Mathematics 106.

\bibitem{Wat} G. N. Watson,
\newblock{\it A Treatise on the theory of Bessel functions.}
\newblock Second Edition. Cambridge University Press, Cambridge (1966.)


\end{thebibliography}
\end{document}